\documentclass[reqno0,12pt]{amsart}
\pdfoutput=1

\usepackage{dsfont}
\usepackage{bbm}
\usepackage[nodate]{datetime}
\usepackage[hmargin=30mm,top=28mm,bottom=28mm,a4paper]{geometry}
\usepackage{color}
\usepackage{subfig}
\usepackage{fancyhdr}
\usepackage{amsmath,amssymb,amsfonts,amsthm}
\usepackage{amstext}
\usepackage{amsmath}
\usepackage{amssymb}
\usepackage{enumerate}
\usepackage{amsbsy}
\usepackage{amsopn}
\usepackage{bbm,amsthm}
\usepackage{amscd}
\usepackage{amsxtra}
\usepackage{todonotes}
\usepackage{soul}

\newtheorem{theorem}{Theorem}[section]
\newtheorem*{theorem*}{Theorem}
\newtheorem{lemma}{Lemma}[section]

\newtheorem{corollary}{Corollary}[section]

\theoremstyle{remark}

\newtheorem{remark}{Remark}[section]

\newcommand{\NN}{\mathds{N}}

\newcommand{\ind}{\mathds{1}}

\newcommand{\EE}{\mathbb{E}}
\newcommand{\PP}{\mathbb{P}}

\DeclareMathOperator{\corr}{corr}
\DeclareMathOperator{\var}{var}
\DeclareMathOperator{\re}{Re}

\title{How many real zeros does a random Dirichlet series have?}
\author{Marco Aymone, Susana Fr\'ometa, Ricardo Misturini}

\begin{document}

\begin{abstract}Let $F(\sigma)=\sum_{n=1}^\infty \frac{X_n}{n^\sigma}$ be a random Dirichlet series where $(X_n)_{n\in\mathbb{N}}$ are independent standard Gaussian random variables. We compute in a quantitative form the expected number of zeros of $F(\sigma)$ in the interval $[T,\infty)$, say $\mathbb{E} N(T,\infty)$, as $T\to1/2^+$. We also estimate higher moments and with this we derive exponential tails for the probability that the number of zeros in the interval $[T,1]$, say $N(T,1)$, is large. We also consider almost sure lower and upper bounds for $N(T,\infty)$. And finally, we also prove results for another class of random Dirichlet series, \textit{e.g.}, when the summation is restricted to prime numbers.
\end{abstract}

\maketitle
\section{Introduction.}
Around 1938, in a series of papers \cite{Littlewood_random_pol_1,Littlewood_random_pol_3,littlewood_random_int_1,littlewood_random_int_2}, Littlewood and Offord proved estimates for the average number of real roots of a random polynomial
$$p(z)=X_0+X_1z+...+X_nz^n,$$
where $(X_j)_{j=0}^n$ are random variables. In 1943, inspired in the first of these papers, Kac~\cite{kac_random_poly} presented a formula for the expected number of these real roots in the Gaussian case. From this formula he deduced that if $n$ is the degree of the random polynomial, and if $(X_j)_{j=0}^n$ are independent standard Gaussian variables, then
$$\EE\mbox{ Number of real roots of }p(z)=\left(\frac{2}{\pi}+o(1)\right)\log n.$$ 

An analogous statement for random variables with other distributions is also true, but this has turned out to be a great challenge in the last century, when, for instance, we consider that $(X_j)_{j=0}^n$ are Rademacher random variables (see this 1956 paper \cite{erdos_random_poly} by Erd\H{o}s and Offord for the Rademacher case, and these other papers \cite{ibragimov_random_poly_1,ibragimov_random_poly_2} by Ibragimov and Maslova for other distributions).

In the past 50 years, this beautiful theory has evolved in deepness and in many perspectives -- here we refer to these papers \cite{vu_repulsion,oanh_local_universality} by Do, Nguyen, Vu  and Nguyen, Vu for a short survey and a nice state of an art on this topic.

In Analytic Number Theory, the location of the zeros of certain analytic functions are of utmost importance. For instance, the location of the zeros of the analytic continuation of the Riemann zeta function 
$$\zeta(s)=\sum_{n=1}^{\infty}\frac{1}{n^s},\;\re(s)>1,$$  
have deep connections with the distribution of prime numbers, where here and throughout this paper, $s$ denotes a complex number $s=\sigma+it$.

The Riemann $\zeta$ function is a particular case of a Dirichlet series, and here we are interested  in the case where  we replace the constant $1$ by random variables, \textit{i.e.}, 
$$F(s)=\sum_{n=1}^\infty\frac{X_n}{n^s},$$
where $(X_n)_{n\in\NN}$ are i.i.d. Gaussian random variables with mean $0$ and variance $1$.

This random Dirichlet series $F(s)$ is, almost surely, convergent if and only if $s$ is in the complex half plane $\re(s)=\sigma>1/2$ due to the Kolmogorov Three-Series Theorem, and to classical results for general Dirichlet series. These series have been studied recently by the authors \cite{aymone_LIL}, where it has been proved a Law of the Iterated Logarithm (LIL) that describes the almost sure fluctuations of $F(\sigma)$ when $\sigma\to1/2^+$ (in the Rademacher case), and by Buraczewski et al. \cite{marynich_random_dirichlet}, where they considered a more general class of this particular random series and proved LIL and other convergence theorems.

A key difference between the zeros of this random Dirichlet series $F(s)$ and that of the Riemann zeta function, is that, $\zeta(s)$ has no real zeros\footnote{Indeed $\zeta(s)=\frac{1}{1-2^{1-s}}\sum_{n=1}^\infty\frac{(-1)^{n+1}}{n^s}$, and this alternating series is well defined for all $\re(s)=\sigma>0$. The fact that $\zeta $ has no real zeros follows from the fact that the sequence $(1/n^\sigma)_{n\in\NN}$ is decreasing and the series is alternated.} in the half plane $\re(s)>0$, while in the random case there are an infinite number of real zeros accumulating at the right of $1/2$, almost surely, see \cite{aymone_real_zeros}.

Throughout this paper we shall specialize on real zeros but in several places we will be looking at $F(s)$ for a given complex number $s$. Our target in this paper is to prove results for $F(\sigma)$ for real $\sigma>1/2$. 

For $1/2<T<U$, $N(T,U)$ denotes the number of real zeros of $F(\sigma)$ in the interval $[T,U]$, where each zero is counted without multiplicity, and $U$ can be either a real number or $\infty$. Since $F(\sigma)$ is an analytic function, $N(T,U)<\infty$ for all $T>1/2$ and $U<\infty$, almost surely. 

As far as we are aware, little attention has been given for zeros of random Dirichlet series in the literature. We found a nice geometric point of view of the expected number of zeros of a general random series of functions by Edelman and Kostlan, see \cite{edelman_roots_random_poly}. For the case of our random Dirichlet series, in \cite{edelman_roots_random_poly} appeared the following formula:
\begin{equation}\label{formula zeros}
	\EE N(T,U)=\frac{1}{\pi}\int_{T}^{U}\sqrt{\frac{d^2}{ds^2}\log \zeta(s)\bigg{|}_{s=2\sigma}}d\sigma.
\end{equation}

We recall that our random Dirichlet series is almost surely convergent at any complex number at the right of $s=1/2$, and divergent at this point and at any complex number at the left of it. By the formula above, we can deduce that for any $1/2<T<U<\infty$, the function inside the integral at the right hand-side of \eqref{formula zeros} is continuous, and hence $\EE N(T,U)<\infty$. However, the point $s=1/2$ is almost surely a singularity of the analytic function given by the random Dirichlet series $F(s)$, c.f. \cite{aymone_real_zeros}. The first aim of this paper is to make it quantitative the formula \eqref{formula zeros} as $T$ gets closer to this singularity at $s=1/2$.
\begin{theorem}\label{teorema principal} There exist $\delta>0$ and constants $c_0$, $(c_n)_{n\geq 2}$ such that, for all $T\in(1/2,1/2+\delta)$:
	\begin{equation}\label{equacao formula quantitativa}
		\mathbb{E} N(T,\infty)=\frac{1}{2\pi}\log\left(\frac{1}{T-1/2}\right)+c_0+\sum_{n=2}^\infty c_n(T-1/2)^{n}.
	\end{equation}
\end{theorem}  

\begin{remark}\label{remark_principal}
	Considering the Laurent expansion of $\zeta$ around its simple pole at $s=1$:
	\begin{equation}\label{equation Laurent expansion}
		\zeta(s)=\frac{1}{s-1}+\sum_{n=0}^\infty (-1)^n\frac{\gamma_n}{n!} (s-1)^n,
	\end{equation}
	where $\gamma_n$ is called the $n$-th Stieltjes constant, it is possible to show that the coefficients $c_n$, in Theorem \ref{teorema principal}, for $n\geq2$, are given by $\frac{1}{\pi}$ times a polynomial $p_n$ with rational coefficients in the variables $(\gamma_n)_{n\geq0}$. In fact, these coefficients can be explicitly computed by formal expansion of power series. For instance, $c_2=\frac{2\gamma_1+\gamma_0^2}{2\pi}$.  
\end{remark}

\subsection{Moment bounds}
Another interesting investigation comes when we consider higher moments $\EE N(T,1)^k$ for a real number $k\geq 1$. We were able to proof the following estimate. 
\begin{theorem}[Moment estimates]\label{theorem moments} There exists a constant $C>0$ such that for all $k\geq 1$ and all $1/2<T<1$,
	$$\EE N(T,1)^k \leq \left(Ck \log\left(\frac{1}{T-1/2}\right) \right)^k.$$
\end{theorem}

As an application of this result, for $C>0$ as in Theorem \ref{theorem moments}, for any fixed $\lambda>2C$, by choosing $k=\lambda/2C$ in Theorem \ref{theorem moments}, we obtain the following Corollary by making a direct usage of Chebyshev's inequality.

\begin{corollary}[Exponential tails] \label{corollary exponential tails} There exist constants $c,C>0$ such that for any $1/2<T<1$ and any $\lambda>C$, 
	$$\PP\left(N(T,1)\geq \lambda \log(1/(T-1/2))\right)\leq \exp(-c\lambda).$$
\end{corollary}

\subsection{Almost sure bounds}
We observe that the zeros of a random polynomial can be very distinct as the degree of the polynomial varies. Here we observe that in our random Dirichlet series case, as $T$ varies, $N(T,\infty)$ is non-decreasing as $T\to1/2^+$. Therefore it becomes natural to consider almost sure limits, and this is the content of our next result. 
\begin{theorem}[Almost sure bounds]\label{theorem almost sure} We have the following almost sure limits:
	\begin{align*}
		&\limsup_{T\to1/2^+}\frac{N(T,\infty)}{\log\left(\frac{1}{T-1/2}\right)\log\log\left(\frac{1}{T-1/2}\right)}<\infty,\\
		&\liminf_{T\to1/2^+}\frac{N(T,\infty)}{\left(\log\left(\frac{1}{T-1/2}\right)\right)^{1/2-\epsilon}}=\infty, \mbox{ for all }\epsilon>0.
	\end{align*}
\end{theorem}

We believe that the upper bound above is close to be optimal, and at the final section we discuss how our methods can be, perhaps, modified in order to get a lower bound of the form $\log(1/(T-1/2))$.

\subsection{More general random Dirichlet series}
We also compute the expected number of real zeros of random Dirichlet series of the form
$$F(\sigma):=\sum_{p}\frac{X_p}{p^{\sigma}},$$
where $p$ runs orderly over an increasing set of positive real numbers $\mathcal{P}:=\{p_1<p_2<...\}$ with $p_1\geq 1$ and $p_n\to\infty$, and $X_p$ are independent standard Gaussian random variables. We assume some regularity in the counting function $\pi(x):=|\{p\leq x:p\in\mathcal{P}\}|$:
\begin{equation}\label{equation definition pi(x)}
	\pi(x)=(1+o(1))x(\log x)^\alpha,\; x\to\infty,
\end{equation}
where $\alpha$ is a real number. As an example, the positive integers satisfy the quantitative statement above with $\alpha=0$, and the prime numbers with $\alpha=-1$, due to the Prime Number Theorem. 

We denote by $N_{\alpha}(T,U)$ the number of zeros in the interval $[T,U]$ of the random series $F(\sigma)$ associated to $\mathcal{P}$ satisfying \eqref{equation definition pi(x)}. Regardless the value of $\alpha$, we have that $F(s)$ converges for all $\re(s)>1/2$, and diverges for all $\re(s)<1/2$, almost surely.

By letting 
$$\zeta_{\alpha}(s):=\sum_{p}\frac{1}{p^s},$$
we see from \cite{edelman_roots_random_poly} that \eqref{formula zeros} generalizes to
\begin{equation}\label{formula zeros alpha}
	\EE N_\alpha(T,U)=\frac{1}{\pi}\int_{T}^{U}\sqrt{\frac{d^2}{ds^2}\log \zeta_\alpha(s)\bigg{|}_{s=2\sigma}}d\sigma.
\end{equation}

It is important to observe that the assumption \eqref{equation definition pi(x)} is not enough to deduce good analytic properties of $\zeta_{\alpha}(s)$ around its singularity at $s=1$. Even so, a qualitative result, weaker in comparison with Theorem \ref{teorema principal}, can be obtained.
\begin{theorem}\label{teorema alpha} As $T\to1/2^+$, we have that
	$$\EE N_{\alpha}(T,\infty)=(1+o(1))\times\begin{cases}&\frac{\sqrt{1+\alpha}}{2\pi}\log\left(\frac{1}{T-1/2}\right), \mbox{ if }\alpha>-1,\\&\frac{1}{\pi}\sqrt{\log\left(\frac{1}{T-1/2}\right)}, \mbox{ if }\alpha=-1, \\& c, \mbox{ if }\alpha<-1, \end{cases}$$
	where $c>0$ is a number that depends on the set $\mathcal{P}$.
\end{theorem}

\section{Notation}
We use the standard notation: 
\begin{enumerate}
	\item $f(x)\ll g(x)$ or equivalently $f(x)=O(g(x))$;
	\item $f(x)=o(g(x))$;
	\item $f(x)\sim g(x)$.
\end{enumerate}
The case (1) is used whenever there exists a constant $C>0$ such that $|f(x)|\leq C |g(x)|$, for all $x$ in a set of numbers. This set of numbers when not specified is the real interval $[L,\infty]$, for some $L>0$, but also there are instances where this set can accumulate at the right or at the left of a given real number, or at complex number. Sometimes we also employ the notation $\ll_\epsilon$ or $O_\epsilon$ to indicate that the implied constant may depends in $\epsilon$. 

In case (2), we mean that $\lim_{x}f(x)/g(x)=0$. When not specified, this limit is as $x\to \infty$ but also can be as $x$ approaches any complex number in a specific direction. 

In case (3), we mean that $f(x)=(1+o(1))g(x)$.

\section{Proof of the main results}
\subsection{The expected number of zeros}
The essence of the proof of Theorem \ref{teorema principal} is the complex analytic theory of the Riemann zeta function as we show below. 
\begin{proof}[Proof of Theorem \ref{teorema principal}] We begin by recalling some well known facts about the Riemann zeta function. Classically defined for $\re(s)>1$ as
	$$\zeta(s)=\sum_{n=1}^{\infty}\frac{1}{n^s},$$
	we have that actually $\zeta$ has analytic continuation to the all complex plane except at $s=1$ where has a simple pole with residue $1$.
	
	In what follows, we will prove eq. \eqref{equacao formula quantitativa} without specifying the values of $c_n$. Afterward, we will indicate how to compute the coefficients $c_n$, $n\geq2$, as stated in Remark \ref{remark_principal}.
	
	We begin by observing that
	$$\frac{d}{ds}\log \zeta(s)=\frac{\zeta'(s)}{\zeta(s)}=-\frac{1}{s-1}+\sum_{n=0}^\infty a_n(s-1)^n,$$
	where the power series above is convergent in the open ball centered at $s=1$ and with radius $3$, since the zero of $\zeta(s)$ closest to $s=1$ is at $s=-2$. Thus, we reach
	\begin{align*}
		\frac{d^2}{ds^2}\log \zeta(s)&=\frac{d}{ds}\frac{\zeta'(s)}{\zeta(s)}\\
		&=\frac{1}{(s-1)^2}+\sum_{n=1}^\infty na_n(s-1)^{n-1}\\
		&=\frac{1}{(s-1)^2}\left(1+A(s)\right),
	\end{align*}
	where $A(s)$ is an analytic function in a open ball centered at $s=1$ and with radius $1$. Moreover, $A(s)=O(|s-1|^2)$ as $s\to1$, and hence, there exists a $\delta>0$ such that $|A(s)|$ does not exceed $1/2$ for all $s$ in an open ball $B$ of center $1$ and radius $\delta$.
	
	Thus, the function $\sqrt{1+A(s)}$ is analytic in this open ball $B$ and has power series representation
	$$\sqrt{1+A(s)}=1+\sum_{n=2}^\infty b_n(s-1)^n.$$
	The first index starting at $n=2$ above is justified by the fact that $A(s)=O(|s-1|^2)$.
	
	Therefore, since for real $1/2<\sigma\leq 1/2+\delta/2$
	$$\sqrt{\frac{d^2}{ds^2}\log \zeta(s)\bigg{|}_{s=2\sigma}}=\frac{1}{2\sigma-1}\sqrt{1+A(2\sigma)},$$ 
	and the power series converges absolutely, the integral of the sum is the sum of integrals:
	\begin{align*}
		\EE N(T,1/2+\delta/2)&=\frac{1}{\pi}\int_T^{1/2+\delta/2}\left(\frac{1}{(2\sigma-1)}+\sum_{n=1}^{\infty}b_{n+1}(2\sigma-1)^n\right)d\sigma\\
		&=\frac{1}{2\pi}\log\left(\frac{1}{T-1/2}\right)+c_0+\sum_{n=2}^{\infty}c_n(T-1/2)^n.
	\end{align*}
	
	Now we will show that $\EE N(1/2+\delta/2,\infty)$ is a constant. Indeed, for $\re(s)>1$ 
	$$\frac{\zeta'(s)}{\zeta(s)}=-\sum_{n=2}^\infty\frac{\Lambda(n)}{n^{s}},$$
	where $\Lambda(n)$ is the classical von Mangoldt function\footnote{The von Mangoldt function is defined as follows: If $n$ is the power of a prime, say $n=p^m$, then $\Lambda(n)=\log p$. If $n$ is not a prime power, then $\Lambda(n)=0$}. Therefore
	$$\frac{d}{ds}\frac{\zeta'(s)}{\zeta(s)}=\sum_{n=2}^\infty\frac{\Lambda(n)\log n}{n^{s}}.$$
	By the general theory of Dirichlet series, $\frac{d}{ds}\frac{\zeta'(s)}{\zeta(s)}$ is a continuous function and always positive in the real interval $[1+\delta,100]$, and hence $\EE N(1/2+\delta/2,100)$ is a real number. Let $L>100$. Since $\sqrt{a+b}\leq \sqrt{a}+\sqrt{b}$ for all $a,b\geq0$, and $0\leq \Lambda(n)\leq \log n$, 
	we have that
	\begin{align*}
		\int_{100}^L \sqrt{\sum_{n=2}^\infty \frac{\Lambda(n)\log n}{n^{2\sigma}}}d\sigma&\leq\int_{100}^L\sum_{n=2}^\infty \frac{\log n}{n^\sigma}d\sigma\\ &\leq\sum_{n=2}^\infty\log n\int_{100}^\infty\exp(-\sigma\log n)d\sigma
		= \sum_{n=2}^\infty\frac{1}{n^{100}}<\infty,
	\end{align*} 
	where the interchange between the integration and summation is justified by the fact that the Dirichlet series converges absolutely for $\sigma$ in the range $[100,L]$, for any large $L>100$. Therefore, the limit 
	$$\lim_{L\to\infty}\EE N(100,L)$$
	exists and is a real number. This completes the proof. \end{proof}

Now we are going to justify that the coefficients $c_n=p_n/\pi$, where $p_n$ is a polynomial with rational coefficients in the variables $(\gamma_n)_{n\geq 0}$.

Before doing that, we recall the following result from Complex Analysis that can easily be obtained from Theorem 3.4, pg. 66 of the book of Lang \cite{Lang}:
\begin{lemma}\label{Lemma composition of series} Let $f$ be analytic in a open ball centered at $w=g(a)$, where $g$ is an analytic function in a open ball centered at $a$. Suppose that
	$$f(z)=\sum_{n=0}^\infty b_n(z-w)^n,\;g(z)=\sum_{n=0}^\infty a_n(z-a)^n.$$
	Then, in a open ball centered at $z=a$, $f(g(z))$ can be represented by a convergent power series given by
	$$f(g(z))=\sum_{n=0}^\infty b_n\left(\sum_{m=0}^\infty a_m(z-a)^m -w\right)^n=\sum_{n=0}^\infty b_n\left(\sum_{m=1}^\infty a_m(z-a)^m \right)^n.$$
\end{lemma}

Working carefully the lemma above, we see that the final power series of the composition is obtained by formally expanding each inner power series at power $n$ and the resulting series is the sum over these expansions.

With that on hand, observe firstly that
\begin{align*}
	\frac{1}{\zeta(s)}&=\frac{s-1}{1-\sum_{n=1}^\infty (-1)^n\frac{\gamma_{n-1}}{(n-1)!}(s-1)^n}:=\frac{s-1}{1-w},
\end{align*}
where $w=w(s)$ is an analytic function at some open ball centered at $s=1$, and $w=O(|s-1|)$ as $s\to 1$. Using the Taylor expansion
$$\frac{1}{1-w}=\sum_{n=0}^\infty w^n,$$
we obtain that at some open ball centered at $s=1$, by Lemma \ref{Lemma composition of series}, $1/\zeta(s)$ around $s=1$ is described by a Taylor series whose coefficients are described by polynomials with rational coefficients in the variables $(\gamma_n)_{n\geq 0}$. The same is true for $\zeta'(s)/\zeta(s)$, since the product of two convergent power series in a ball are again a convergent power series in a, perhaps, smaller ball. And moreover
$$\left(\sum_{n=0}^\infty a_n s^n\right)\left(\sum_{n=0}^\infty b_n s^n\right)=\sum_{n=0}^\infty \left(\sum_{\substack{u,v\geq 0\\ u+v=n }}a_ub_v\right) s^n.$$

Now, $\frac{d}{ds} \zeta'(s)/\zeta(s) $ is described by a Laurent series whose coefficients are described by polynomials with rational coefficients in the variables $(\gamma_n)_{n\geq 0}$, and the same is true for $A(s)$ defined above and consequently for $\sqrt{1+A(s)}$, since the power series of $\sqrt{1+z}$ in the variable $z$ has rational coefficients. The last step was to 
integrate $\frac{1}{2\sigma-1}\sqrt{1+A(2\sigma)}$, and this keeps the target property. This justifies Remark~\ref{remark_principal}.

\subsection{Moment bounds}
The proof is based on the following inequality involving the number of zeros of an analytic function and its maximal value in circles: Let $F(s)$ be analytic in a domain containing the disc $|s|\leq R$, let $M$ be the maximal value of $|F|$ on this disc, and assume that $F(0)\neq 0$. Then, for $r<R$, the number of zeros of $F$ in the disc $|s|\leq r$ does not exceed 
\begin{equation}\label{equacao desigualdade zeros}
	\frac{\log (M/F(0))}{\log(R/r)}.
\end{equation}
A proof of this can be found at the book of Montgomery and Vaughan \cite{montgomerylivro}, pg. 168.

Now we begin the proof of Theorem \ref{theorem moments} with the following lemma concerning bounds of Dirichlet series at different points.
\begin{lemma}\label{lemma bounds d series} Let $F(s)=\sum_{n=1}^\infty X_nn^{-s}$ be a Dirichlet series convergent for $\re(s)>1/2$. Let $\re(s)>1/2$ and $\sigma_1\in(1/2,\re(s))$. Let $A(t)=\sum_{n\leq t}X_n n^{-\sigma_1}$. Then 
	$$|F(s)|\leq \frac{|s-\sigma_1|}{\re(s)-\sigma_1}\sup_{t\geq 1}|A(t)|.$$
\end{lemma}
\begin{proof}
	We have that
	$$F(s)=\sum_{n=1}^\infty \frac{X_n}{n^{\sigma_1}}\cdot\frac{1}{n^{s-\sigma_1}}.$$
	Now we write the sum above as a Riemann-Stieltjes integral:
	$$F(s)=\int_{1^-}^\infty \frac{1}{t^{s-\sigma_1}}dA(t).$$
	Using integration by parts, we reach:
	$$|F(s)|= \left|(s-\sigma_1)\int_{1}^\infty \frac{A(t)}{t^{s-\sigma_1+1}}dt\right|\leq \frac{|s-\sigma_1|}{\re(s)-\sigma_1}\sup_{t\geq 1}|A(t)|.$$
\end{proof}

Now we recall a classical inequality for sums of independent random variables:

\textit{Levy's maximal inequality}: Let $X_1,...,X_n$ be independent random variables. Then for all $t\geq0$
\begin{equation}\label{equation Levy finita}
	\PP\bigg{(}\max_{1\leq m \leq n}\bigg{|}\sum_{k=1}^m X_k \bigg{|}\geq t \bigg{)}\leq 3 \max_{1\leq m \leq n}\PP\bigg{(}\bigg{|}\sum_{k=1}^m X_k \bigg{|}\geq \frac{t}{3} \bigg{)}.
\end{equation}
A proof of this can be found in the nice book of Pe\~na and Giné, \cite{pena_and_gine_decoupling}, pg. 4. In what follows we will need an infinite version of the inequality above:

\begin{equation}\label{equation Levy maximal}
	\PP\bigg{(}\sup_{m\geq 1}\bigg{|}\sum_{k=1}^m X_k \bigg{|}\geq t \bigg{)}\leq 3 \sup_{m\geq 1}\PP\bigg{(}\bigg{|}\sum_{k=1}^m X_k \bigg{|}\geq \frac{t}{3} \bigg{)}.
\end{equation}
To obtain this, we observe that the event inside the probability in the left hand-side of \eqref{equation Levy finita} increases as $n\to\infty$ to the event in the left hand-side of \eqref{equation Levy maximal}. Therefore, by the continuity of probabilities, to obtain \eqref{equation Levy maximal} we only need to make the limit $n\to\infty$ in \eqref{equation Levy finita}. 

Another inequality we will need is the following: Let $x_1,...,x_R$ be positive real numbers. Then, for all $k\geq 1$
\begin{equation}\label{equation desigualdade media aritimetica}
	(x_1+...+x_R)^k\leq R^{k-1}(x_1^k+...+x_R^k).
\end{equation}
The proof of this inequality can be made by the following argument: Let $X$ be a random variable with uniform distribution over $\{x_1,...,x_R\}$. Then the above inequality is just the moment bound $\EE X \leq (\EE X^k)^{1/k}$.

We continue with the following:

\begin{lemma}\label{lemma esperanca log a k} Let $F(s)$ be our random Dirichlet series. Let $0<\delta\leq 1/2$, and $C_0$ be a circle with center $\sigma_0=1/2+5\delta/4$ and radius $\delta/4$. Let $C_1$ be a circle with same center $\sigma_0$ but with radius $\delta/2$.
	Let $M=\max_{s\in C_1}|F(s)|$. Then there exists a constant $C>0$ that does not depend on $\delta$ such that for all $k\geq 1$
	$$\EE \left|\log (M/F(\sigma_0))\right|^k \leq C^kk^k.$$
\end{lemma}

\begin{proof}
	Let $\sigma_1=1/2+\delta/2$. The maximal value that $|s-\sigma_1|$ can attain for $s\in C_1$ is below $2 \delta$, and the minimal value that $\re(s)-\sigma_1$ can attain is $\delta/4$. Therefore, by Lemma \ref{lemma bounds d series} we have that almost surely
	\begin{equation}\label{equation inequality for M}
		M\leq 8 \sup_{t\geq 1}|A(t)|,
	\end{equation}
	where $A(t)=\sum_{n\leq t}X_n n^{-\sigma_1}$.

	Observe that $F(\sigma_1)$ has variance $\sim \frac{1}{\delta}$, and $F(\sigma_0)$ has variance $\sim \frac{2}{5\delta}$.
	Therefore
	\begin{align*}
		\left|\log (M/F(\sigma_0))\right|^k&=\left|\log \left(\frac{\sqrt{\delta} M}{\sqrt{5\delta} F(\sigma_0)/\sqrt{2}}\right)+\log(\sqrt{5/2})\right|^k\\
		&\ll 3^{k-1}(|\log(\sqrt{\delta} M)|^k+|\log(\sqrt{5\delta} F(\sigma_0)/\sqrt{2})|^k+\log^k(\sqrt{5/2})),
	\end{align*}
	where we used the inequality \eqref{equation desigualdade media aritimetica} in the last step above.
	
	Now we will estimate $\EE|\log(\sqrt{\delta} M)|^k$. The estimate of 
	$\EE|\log(\sqrt{5\delta} F(\sigma_0)/\sqrt{2}|^k$ can be done analogously. We split $\sqrt{\delta} M=\sqrt{\delta} M\ind_{[0\leq \sqrt{\delta} M <1]}+\sqrt{\delta} M\ind_{[\sqrt{\delta} M\geq 1]}$.
	
	\noindent \textit{Estimate for} $\EE |\log (\sqrt{\delta} M)|^k\ind_{[\sqrt{\delta} M \geq 1]}$. We split $\ind_{[\sqrt{\delta} M \geq 1]}=\sum_{n=1}^\infty\ind_{[3^{n-1}\leq \sqrt{\delta} M < 3^n]}$. In the event $[3^{n-1}\leq \sqrt{\delta} M \leq 3^n]$, we have that $|\log (\sqrt{\delta} M)|^k\leq (n\log 3)^k$. Further, by inequality \eqref{equation inequality for M}, we have that
	$$\PP\left(3^{n-1}\leq \sqrt{\delta} M \leq 3^n\right)\leq\PP\left(\sqrt{\delta} M\geq 3^{n-1}\right)\leq \PP\left(8 \sup_{t\geq 1}\sqrt{\delta} |A(t)|\geq 3^{n-1}\right).$$
	By Levy's maximal inequality \eqref{equation Levy maximal} applied for $A(t)$, we reach
\begin{align*}
\PP\left(3^{n-1}\leq \sqrt{\delta} M \leq 3^n\right)&\leq 3\sup_{t\geq 1}\PP\left(|\sqrt{\delta} A(t)|\geq 3^{n-4}\right) \\
&= 3\PP\left(|\sqrt{\delta} F(\sigma_1)|\geq 3^{n-4}\right),
\end{align*}
where the equality above is justified by the fact that for each $t\geq 1$, $A(t)$ has Gaussian distribution with mean $0$ and with variance increasing to the variance of $F(\sigma_1)$ as $t\to\infty$. But now we recall that $\sqrt{\delta} F(\sigma_1)$ has Gaussian distribution with variance tending to $1$ as $\delta\to0^+$, and certainly above a small fixed $c>0$ and below a large fixed $d>0$, for all $0<\delta\leq 1/2$. Hence, we are allowed to use a rough bound 
	$\PP\left(3^{n-1}\leq \sqrt{\delta} M \leq 3^n\right)\ll 3^{-n}$. Therefore
	$$\EE |\log (\sqrt{\delta} M)|^k\ind_{[\sqrt{\delta} M \geq 1]}\ll \sum_{n=1}^{\infty}\frac{n^k}{3^n}=\sum_{n=1}^{\infty}\frac{n^ke^{-n}}{(3/e)^n}\ll \sup_n n^ke^{-n}\leq k^ke^{-k},$$
	where in the last step above we just used elementary tools from calculus.
	
	\noindent \textit{Estimate for} $\EE |\log (\sqrt{\delta} M)|^k\ind_{[0\leq \sqrt{\delta} M\leq 1]}$. We begin by observing that the if $M=0$, then by analyticity $F(s)=0$ everywhere and hence that all Gaussians $(X_n)_{n\geq 1}$ are identically $0$, and this happens with probability $0$. So we can work with $\EE |\log (\sqrt{\delta} M)|^k\ind_{[0< \sqrt{\delta} M \leq 1]}$. As before, we split 
	$\ind_{[0<\sqrt{\delta} M \leq 1]}=\sum_{n=1}^\infty\ind_{[3^{-n}< \sqrt{\delta} M \leq 3^{-n+1}]}.$
	In the event $[3^{-n}< \sqrt{\delta} M \leq 3^{-n+1}]$, we have that 
	$$|\log (\sqrt{\delta} M)|^k\leq (n\log 3)^k.$$ Let $u=1/2+3\delta/4$ be the leftmost point of the circle $C_1$. We thus have
	$$\PP\left(3^{-n}< \sqrt{\delta} M \leq 3^{-n+1}\right)\leq\PP\left(\sqrt{\delta} M \leq 3^{-n+1}\right)\leq \PP\left(|\sqrt{\delta} F(u)| \leq 3^{-n+1}\right).$$
	Now observe that $\sqrt{\delta} F(u)$ is Gaussian with variance bounded below by some $c>0$ and above by some $d>0$ in the interval $0<\delta\leq 1/2$, and hence, for sufficiently small $\epsilon>0$, 
	$$\PP(|\sqrt{\delta} F(u)|\leq \epsilon)\ll \epsilon.$$
	Thus, as in the previous case above
	$$\EE |\log (\sqrt{\delta} M)|^k\ind_{[0\leq \sqrt{\delta} M \leq 1]}\ll \sum_{n=1}^{\infty}\frac{n^k}{3^n} \ll k^ke^{-k}, $$
	and this completes the proof of the lemma.
\end{proof}

We are ready to:
\begin{proof}[Proof of Theorem \ref{theorem moments}] We let $T=1/2+\delta$ for some $0<\delta<1/2$. Let $C_0^{(0)}$ be a circle with center $\sigma_0=1/2+5\delta/4$ and radius $\delta/4$. Let $C_1^{(0)}$ be a circle with same center $\sigma_0$ but with radius $\delta/2$. Call the rightmost point of $C_0^{(0)}$ by $T_1=1/2+3\delta/2$. Isolating $\delta$ in terms of $T$, we can write $T_1=1/2+(T-1/2)(3/2)$. Define $\delta_1=T_1-1/2$, $C_0^{(1)}$ the circle with leftmost point at $s=T_1$ and with radius $\delta_1/4$. The rightmost point of $C_0^{(1)}$ is $T_2=1/2+(T-1/2)(3/2)^2$. Define $C_1^{(1)}$ to be a circle with same center as $C_0^{(1)}$ but with radius $\delta_1/2$. By continuing this process, we obtain a sequence of pairs of concentric circles $C_0^{(n)}$ and $C_1^{(n)}$ with same proportions of $C_0^{(0)}$ and $C_1^{(0)}$ so that Lemma \ref{lemma esperanca log a k} can be used in each one of this pair of circles.
	
	Set $T_0=T$. The number of zeros $N(T,1)$ can be split as 
	$$N(T,1)\leq N(T_0,T_1)+N(T_1,T_2)+N(T_2,T_3)+...+N(T_{R-1},T_R),$$
	where $R$ is defined as the smallest positive integer $n$ so that $T_n\geq 1$. Since $T_n=1/2+(T-1/2)(3/2)^n$, we have that $R\ll \log(1/(T-1/2))$.
	
	By inequality \eqref{equation desigualdade media aritimetica}, we have that
	$$N(T,1)^k\leq R^{k-1}\sum_{n=0}^{R-1}N^k(T_{n},T_{n+1}).$$
	Say that the center of $C_{0}^{(n)}$ is $\sigma_0^{(n)}$ and that $M_n=\max_{s\in C_1^{(n)}}|F(s)|$. Using the inequality \eqref{equacao desigualdade zeros}, we see that
	$$N(T_n,T_{n+1})\ll \log(M_n/F(\sigma_0^{(n)})).$$
	Thus combining \eqref{equation desigualdade media aritimetica} with Lemma \ref{lemma esperanca log a k}, we obtain constants $C,D>0$ such that
	$$\EE N(T,1)^k\leq R^{k-1}\sum_{n=0}^{R-1}C^k k^k\leq C^kk^kR^k\leq D^kk^k\log^k(1/(T-1/2)),$$
	and this finishes the proof. \end{proof}

\subsection{Almost sure bounds}

We begin with the
\begin{proof}[Proof of the upper bound] Let $T_n=1/2+1/2^n$. 
	
We use Corollary \ref{corollary exponential tails} with $\lambda_n=D\log\log(1/(T_n-1/2))$, for some constant $D>0$:
	$$\PP(N(T_n,1)\geq \lambda_n \log(1/(T_n-1/2)))\ll \frac{1}{n^{Dc}}.$$
	So, if $Dc>1$, the probabilities above are summable and hence the Borel-Cantelli Lemma is applicable, that is, almost surely for all $n$ sufficiently large
	$$\frac{N(T_n,1)}{D\log(1/(T_n-1/2)))\log \log(1/(T_n-1/2)))}\leq 1.$$
	Observe that $N(T)$ is non-decreasing as $T\to1/2^+$. Now, for $T_{n}\leq T \leq T_{n-1}$, we have that 
	\begin{align*}
		&\frac{N(T,1)}{D\log(1/(T-1/2)))\log \log(1/(T-1/2)))}\\
		&\leq \frac{D\log(1/(T_n-1/2)))\log \log(1/(T_n-1/2)))}{D\log(1/(T_{n-1}-1/2)))\log \log(1/(T_{n-1}-1/2)))}\\
		&\leq 4.
	\end{align*}
	To pass the above estimate to $N(T,\infty)$, we just observe that any Dirichlet series has a half plane inside its region of absolute convergence in which there are no zeros. A proof of this can be found at the book of Apostol \cite{apostol} pg. 227. In our case, our random Dirichlet series converges absolutely and almost surely in the half plane $\re(s)>1$. To complete the argument, we see that since a Dirichlet series is an analytic function, it cannot have and infinite number of real zeros between $s=1$ and this half plane where it does not vanishes, unless that it vanishes identically, which is not the case. \end{proof}

Now we continue with the proof of Theorem \ref{theorem almost sure}\\
\noindent \textit{The lower bound.} The proof of the lower bound will be divided in some steps. Our idea is to consider the following quantities: Let $\sigma_n=1/2+1/2^n$ and define
		\begin{align*}
			S_{+}(R)&=\sum_{n\leq R}\ind_{[F(\sigma_n)>0]},\\
			S_{-}(R)&=\sum_{n\leq R}\ind_{[F(\sigma_n)<0]}.
	\end{align*}
Thus $S_{+}(R)$ counts the number of times that $F$ is positive along the sequence $\sigma_n$, and $S_{-}(R)$ the number of times that $F$ is negative along the same sequence.

In what follows, we will prove a quantitative Law of large numbers for $S_{+}(R)$ and for $S_{-}(R)$. We need firstly the following result.

\begin{lemma}\label{lemma calculo da correlacao} Let $F(\sigma)$ be our random Dirichlet series and $\sigma_n=2^{-1}+2^{-n}$. Then there exists a constant $C>0$ such that
	$$|\corr(F(\sigma_k),F(\sigma_l))|\leq \frac{C}{\sqrt{2}^{|k-l|}}.$$
\end{lemma}
\begin{proof}
	We have that
	$$\EE F(\sigma_k)F(\sigma_l)=\sum_{n=1}^\infty\frac{1}{n^{\sigma_k+\sigma_l}}=\zeta(\sigma_k+\sigma_l)\sim \frac{1}{2^{-k}+2^{-l}}.$$
	On the other hand, 
	$$\EE F(\sigma)^2=\zeta(2\sigma)\sim\frac{1}{2\sigma-1}.$$
	Therefore
	$$\corr(F(\sigma_k),F(\sigma_l))\sim\frac{\sqrt{2^{-k-l+2}}}{2^{-k}+2^{-l}}=\frac{2}{2^{(l-k)/2}+2^{(k-l)/2}}\ll \frac{1}{\sqrt{2}^{|k-l|}}.$$
\end{proof}
\begin{lemma}\label{lemma correlacao de eventos} Let $F(\sigma)$ be our random Dirichlet series and $\sigma_n=2^{-1}+2^{-n}$. There exists a constant $C>0$ such that for all $k$ and $l$
	$$|\corr(\ind_{[F(\sigma_k)>0]},\ind_{[F(\sigma_l)>0]})|\leq \frac{C}{\sqrt{2}^{|k-l|}},$$
and 
	$$|\corr(\ind_{[F(\sigma_k)<0]},\ind_{[F(\sigma_l)<0]})|\leq \frac{C}{\sqrt{2}^{|k-l|}}.$$
\end{lemma}
\begin{proof} We begin by observing that $F(\sigma_k)$ and $F(\sigma_l)$ have joint Gaussian distribution. This can be verified by checking that any linear combination of them has a Gaussian distribution, see the book of Shiryaev \cite{shiryaev} pg. 301. In our case, for any real numbers $a$ and $b$:
	$$aF(\sigma_k)+bF(\sigma_l)=\sum_{n=1}^\infty X_n\left(\frac{a}{n^{\sigma_k}}+\frac{b}{n^{\sigma_l}}\right),$$
	and since $(X_n)_{n\in\NN}$ are i.i.d. standard Gaussians, we reach that the distribution of $aF(\sigma_k)+bF(\sigma_l)$ also is Gaussian.
	
	Now, a simple calculation shows that
	$$\corr(\ind_{[F(\sigma_k)>0]},\ind_{[F(\sigma_l)>0]})=4\PP(F(\sigma_k)>0,F(\sigma_l)>0)-1.$$
	Let $X=F(\sigma_k)/\sqrt{\var F(\sigma_k)}$ and $Y=F(\sigma_l)/\sqrt{\var F(\sigma_l)}$. Thus, the probability in the right-handside above is $\PP(X>0,Y>0)$. Observe that $X$ and $Y$ are standard Gaussians with correlation $\rho$, say. Let $Z$ be another standard Gaussian random variable independent of $X$. Thus $(X,Y)$ has the same distribution of $(X,\rho X+\sqrt{1-\rho^2}Z)$.
	With this we reach
	\begin{align*}
		\PP(X>0,Y>0)&=\PP\left(X>0,Z>-\frac{\rho}{\sqrt{1-\rho^2}}X\right)\\
		&=\frac{1}{4}+\PP\left(X>0,0>Z>-\frac{\rho}{\sqrt{1-\rho^2}}X\right)\\
		&:=\frac{1}{4}+f(\rho).
	\end{align*}
	Now we see that $f(\rho)$ is the probability that the pair $(X,Z)$ lies in a sector with angle 
	$$\theta=\tan^{-1}\left(\frac{\rho}{\sqrt{1-\rho^2}}\right).$$ 
	Since the distribution of $(X,Z)$ is invariant by rotations, we have that $f(\rho)=\theta/2\pi$.
	So our target correlation is
	$$\corr(\ind_{[F(\sigma_k)>0]},\ind_{[F(\sigma_l)>0]})=\frac{2}{\pi}\tan^{-1}\left(\frac{\rho}{\sqrt{1-\rho^2}}\right).$$

	Using the fact that $\tan^{-1}(x)=\int_{0}^x\frac{dt}{1+t^2}\leq x$ and Lemma \ref{lemma calculo da correlacao}, we complete the proof of the lemma in the first case. The second case can be done analogously. \end{proof}

\begin{lemma}\label{lemma law of large numbers} Let $S_{+}(R)$ and $S_{-}(R)$ be as above. Then, for all $\epsilon>0$
		$$S_{+}(R),S_{-}(R)=\frac{R}{2}+O_\epsilon(R^{1/2+\epsilon}), \,a.s.$$
\end{lemma}
\begin{proof}
	The proof of this lemma is a direct application of a result in Probability theory for weak dependencies that says the following: Let $Y_n$ be a sequence of square-integrable random variables such that
	$$\sup_{n\geq 1}|\corr(Y_n,Y_{n+k})|\leq \rho_k,$$
	for some constants $\rho_k$ such that 
	\begin{equation}\label{equation soma rho_k}
		\sum_{k=1}^{\infty}\rho_k<\infty.
	\end{equation}
	If for some increasing sequence of positive real numbers $b_n$ we have that
	\begin{equation}\label{equation soma variancias}
		\sum_{n=1}^\infty\frac{\var(Y_n)(\log n)^2}{b_n^2}<\infty,
	\end{equation}
	then the series 
	$$\sum_{n=1}^\infty\frac{Y_n-\EE Y_n}{b_n}$$
	converges almost surely.
	
	A proof of this can be found at the book of Stout \cite{stout}, pg. 28.
	
In our case, we apply the result above for $Y_n=\ind_{[F(\sigma_n)>0]}$, and $Y_n=\ind_{[F(\sigma_n)<0]}$.  We have that \eqref{equation soma rho_k} is satisfied by Lemma \ref{lemma correlacao de eventos}. Now we take $b_n=n^{1/2+\epsilon}$. Since $0\leq Y_n\leq 1$, \eqref{equation soma variancias} also is satisfied. Thus we get convergence of the series
$$\sum_{n=1}^\infty\frac{\ind_{[F(\sigma_n)>0]}-1/2}{n^{1/2+\epsilon}},\,\sum_{n=1}^\infty\frac{\ind_{[F(\sigma_n)<0]}-1/2}{n^{1/2+\epsilon}}.$$
Now we recall a particular case of Kroeneckers' Lemma (see the book of Shiryaev \cite{shiryaev}, pg. 390): For any sequence of real numbers $(a_n)_n$ and $\sigma>0$, if the series $\sum_{n=1}^\infty a_n n^{-\sigma}$ converges, then the partial sums $\sum_{n\leq x}a_n=o(x^\sigma)$. By applying this result to the random variable $a_n=\ind_{[F(\sigma_n)>0]}-1/2$ or $a_n=\ind_{[F(\sigma_n)<0]}-1/2$, we obtain the target result.
\end{proof}
We are ready to the
\begin{proof}[Proof of the Lower bound]
	We see that if both $S_{+}(R+V)-S_{+}(R)\geq 1$ and $S_{-}(R+V)-S_{-}(R)\geq 1$, then $F(\sigma)$ has at least one sign change in the interval $[\sigma_{R+V},\sigma_R]$, and consequently has a zero in this interval.
	
	By Lemma \ref{lemma law of large numbers}
	\begin{align*}
		S_{+}(R+V)-S_{+}(R)&=\frac{R+V}{2}+O((R+V)^{1/2+\epsilon})-\left(\frac{R}{2}+O(R^{1/2+\epsilon})\right)\\
		&=\frac{V}{2}+O((R+V)^{1/2+\epsilon}),
\end{align*}
	and the same holds for $S_{-}$.
	
	In order to maximize the counting of number of zeros we need $V$ as small as possible. But the equality above says that to guarantee a zero in the interval $[\sigma_{R+V}, \sigma_R]$ we need $V$ a bit larger than $O_\epsilon((R+V)^{1/2+\epsilon})$, where the implicit constant in this $O_{\epsilon}$ term might depend in $\epsilon$ and could be random. So, we seek for a sequence $R_n$ such that, if $\epsilon>0$ is small, then $(R_{n+1}-R_n)/R_{n+1}^{1/2+\epsilon}\to\infty$.
	
	Indeed such property is satisfied by choosing $R_n=[n^{2+8\epsilon}]$ (here $[x]$ stands for the integer part of $x$), since by the mean value theorem for differentiable functions, for some $n\leq \theta_n\leq n+1$,
	$$R_{n+1}-R_n=(2+8\epsilon)\theta_n^{1+8\epsilon},$$
	and then
$$\frac{R_{n+1}-R_n}{R_{n+1}^{1/2+\epsilon}}\sim (2+8\epsilon)n^{2\epsilon-8\epsilon^2}.$$

	So we showed that almost surely for all $n$ sufficiently large, there is a zero of $F(\sigma)$ in the interval $[\sigma_{R_{n+1}},\sigma_{R_{n}}]$.
	The number of these subintervals has size proportional to $R^{1/2-\epsilon'}$ for some new small $\epsilon'>0$. Indeed, if $n$ is the largest positive integer $k$ such that $R_k\leq R$, then $n\sim R^{1/(2+8\epsilon)}$.
	
	Thus, for any $\epsilon>0$, almost surely for all $R$ sufficiently large there is at least $(1+o(1))R^{1/2-\epsilon}$ zeros of $F(\sigma)$ in the interval $[1/2+1/2^R,1]$ and hence at least a quantity $\gg (\log(1/(T-1/2)))^{1/2-\epsilon}$ of zeros in the interval $[T,1]$.  \end{proof}

\subsection{More general random Dirichlet series}

\begin{proof}[Proof of Theorem \ref{teorema alpha}] Just as in Theorem \ref{teorema principal}, we have that
	$$\EE N_\alpha(T,U)=\frac{1}{\pi}\int_{T}^U\sqrt{\frac{\zeta_{\alpha}''(2\sigma)}{\zeta_\alpha(2\sigma)}-\left(\frac{\zeta_\alpha'(2\sigma)}{\zeta_\alpha(2\sigma)}\right)^2}d\sigma.$$
	
	Thus, we need to estimate, as $\sigma\to1/2^+$, quantities of the form
	$$\zeta_{\alpha}^{(\beta)}(2\sigma)=\sum_{p}\frac{(\log p)^\beta}{p^{2\sigma}}=\int_{2}^\infty \frac{(\log x)^\beta}{x^{2\sigma}}d\pi(x)+O(1),$$
	where $\beta=0,1,2$ and the last integral above is in the Riemann-Stieltjes sense. Integration by parts gives, as $\sigma\to{1/2^+}$:
	$$\zeta_{\alpha}^{(\beta)}(2\sigma)=(2\sigma+o(1))\int_{2}^\infty \frac{\pi(x)(\log x)^\beta}{x^{2\sigma+1}}dx=(1+o(1))\int_{2}^\infty \frac{(\log x)^{\alpha+\beta}}{x^{2\sigma}}dx.$$
	\begin{lemma}\label{lemma integral J} Let $\gamma$ be a real number.
		Then, as $\sigma\to1/2 ^+$: 
		$$J(\gamma,\sigma):=\int_{2}^\infty \frac{(\log x)^{\gamma}}{x^{2\sigma}}dx=(1+o(1))\begin{cases}& \frac{\Gamma(\gamma+1)}{(2\sigma-1)^{\gamma+1}}, \gamma>-1\\ & \log\left(\frac{1}{\sigma-1/2}\right), \gamma=-1 \\ & c(\gamma), \gamma<-1,  \end{cases},$$
		where $\Gamma$ is the classical Euler's Gamma function, and $c(\gamma)$ is a constant that depends on $\gamma$.
	\end{lemma}
	\begin{proof}[Proof of Lemma \ref{lemma integral J}]
		The case $\gamma<-1$ is easy. Let then $\gamma>-1$. Let $u=(2\sigma-1)\log x$. Then
		\begin{align*}
			J(\gamma,\sigma)&=\frac{1}{(2\sigma-1)^{1+\gamma}}\int_{(2\sigma-1)\log 2}^\infty u^{(1+\gamma)-1} e^{-u}du\\
			&=(1+o(1))\frac{\Gamma(\gamma+1)}{(2\sigma-1)^{1+\gamma}}.
		\end{align*}
		In the case that $\gamma=-1$,
		\begin{align*}
			J(-1,\sigma)&=\int_{(2\sigma-1)\log 2}^\infty \frac{ e^{-u}}{u}du\\
			&=\int_{(2\sigma-1)\log 2}^{1/100} \frac{ e^{-u}}{u}du+O(1)\\
			&=\int_{(2\sigma-1)\log 2}^{1/100} \frac{ 1+O(u)}{u}du+O(1)\\
			&=\log\left(\frac{1}{\sigma-1/2}\right)+O(1).
		\end{align*}
		This proves the lemma. \end{proof}
	
	Now we continue with the proof of Theorem \ref{teorema alpha}. We begin by estimating $\EE N_\alpha(T,1/2+\delta)$ for some small $\delta>0$.
	
	\noindent{Case $\alpha>-1$}. In this case, by Lemma \ref{lemma integral J} we have that
	\begin{align*}
		\EE N_\alpha(T,1/2+\delta)=\frac{1}{\pi}\int_{T}^{1/2+\delta} \bigg{(}&\frac{(1+o(1))\Gamma(3+\alpha)(2\sigma-1)^{-(3+\alpha)}}{(1+o(1))\Gamma(1+\alpha)(2\sigma-1)^{-(1+\alpha)}}\\
		-&\left(\frac{(1+o(1))\Gamma(2+\alpha)(2\sigma-1)^{-(2+\alpha)}}{(1+o(1))\Gamma(1+\alpha)(2\sigma-1)^{-(1+\alpha)}}\right)^{2}\bigg{)}^{1/2}d\sigma.
	\end{align*}
	Due to the property that $\Gamma(z+1)=z\Gamma(z)$, the last expression simplifies to
	\begin{align*}
		\EE N_\alpha(T,1/2+\delta)=&(1+o(1))\frac{\sqrt{1+\alpha}}{\pi}\int_{T}^{1/2+\delta}\frac{1}{2\sigma-1}d\sigma\\
		&=(1+o(1))\frac{\sqrt{1+\alpha}}{2\pi}\log\left(\frac{1}{T-1/2}\right).
	\end{align*}

	\noindent{Case $\alpha=-1$}. In this case
	\begin{align*}
		\EE N_{-1}(T,1/2+\delta)=&\frac{1}{\pi}\int_{T}^{1/2+\delta} \bigg{(}\frac{(1+o(1))\Gamma(2)(2\sigma-1)^{-2}}{\log(1/(2\sigma-1))}\bigg{)}^{1/2}d\sigma\\
		&=\frac{(1+o(1))}{2\pi}\int_{2T-1}^{2\delta}\frac{1}{x\sqrt{-\log x}}dx\\
		&=\frac{(1+o(1))}{2\pi}\int_{-\log(2T-1)}^{-\log(2\delta)}-v^{-1/2}dv\\
		&=\frac{(1+o(1))}{\pi}\sqrt{\log\left(\frac{1}{T-1/2}\right)}.
	\end{align*}
	
	\noindent{Case $\alpha<-1$}. In this case we have that $\zeta_\alpha(2\sigma)=c+o(1)$, as $\sigma\to1/2^+$, for some $c>0$. The proof of this case follows the idea of the previous ones, in which the function $J(\gamma,\sigma)$ of Lemma \ref{lemma integral J} is analyzed. For that, we will have to divide our proof in the cases $-2<\alpha<-1$, $\alpha=-2$, $-3<\alpha<-2$, $\alpha=-3$ and $\alpha<-3$. We will present the details only for the case $-2<\alpha<-1$. The other cases can be treated similarly.
	
	Let then $-2<\alpha<-1$. We have that
	\begin{align*}
		\EE N_\alpha(T,1/2+\delta)=\frac{1}{\pi}\int_{T}^{1/2+\delta} &\bigg{(}\frac{(1+o(1))\Gamma(3+\alpha)(2\sigma-1)^{-(3+\alpha)}}{c+o(1)}\\
		-&\left(\frac{(1+o(1))\Gamma(2+\alpha)(2\sigma-1)^{-(2+\alpha)}}{c+o(1)}\right)^{2}\bigg{)}^{1/2}d\sigma.
	\end{align*}
	The function inside the square-root above behaves, as $\sigma\to1/2^+$, as a constant times 
	$$(2\sigma-1)^{-\frac{3+\alpha}{2}}.$$
	Apart from the fact that this function blows as $\sigma\to1/2^+$, we have that the exponent $(3+\alpha)/2$ lies in the interval $(1/2,1)$, and hence the hole function is integrable in the interval $[1/2^+,1]$.
	
	Now we will show that in any case, $\EE N(1/2+\delta,\infty)$ is a real number. The function $\zeta_\alpha(\sigma)$ converges absolutely for $\sigma>1$, and hence it is an analytic function. Further, $\zeta_\alpha(\sigma)$ is a series of positive numbers, and hence $\zeta_{\alpha}(\sigma)\neq 0$ for all $\sigma>1$. Hence $\EE N_{\alpha}(1/2+\delta,100)$ is the definite integral of a continuous function, a real number. Consider now
	$$F(\sigma)=\sum_{n=1}^\infty\frac{X_{p_n}}{p_n^\sigma}=p_1^{-\sigma}\sum_{n=1}^\infty\frac{X_{p_n}}{(p_n/p_1)^{\sigma}}:=p_1^{-\sigma}G(\sigma).$$
	Thus, $F(\sigma)$ share same zeros with $G(\sigma)$. Now we can write 
	$$\zeta_{\mathcal{Q}}(\sigma):=\sum_{q\in\mathcal{Q}}\frac{1}{q^\sigma},$$
	where $\mathcal{Q}=\{q_1=1<q_2<q_3<...\}$ and $q_n=p_n/p_1$, for all $n$.
	Thus, $\zeta_{\mathcal{Q}}(\sigma)>1$ for all $\sigma>1$ and $\lim_{\sigma\to\infty}\zeta_{\mathcal{Q}}(\sigma)=1$. Hence
	
	\begin{align*}
		\EE N_\alpha(100,L)&\ll \int_{100}^L\sqrt{\frac{\zeta_{\mathcal{Q}}''(2\sigma) }{\zeta_{\mathcal{Q}}(2\sigma)}}d\sigma\\
		&\ll \int_{100}^L\sqrt{\zeta_{\mathcal{Q}}''(2\sigma)} d\sigma\\
		&\ll \int_{100}^L\sum_{q>1}\frac{\log q}{q^\sigma} d\sigma\\
		&\ll \sum_{q>1}\frac{1}{q^{100}}\\
		&<\infty.
	\end{align*}
	This shows that $\EE N(100,\infty)$ is a real number, and this ends the proof. \end{proof}

\section{Concluding remarks}
\subsection{Almost sure lower bound} We believe that our almost sure lower bound is far from being optimal, and we include it here for completeness. We describe an approach that could perhaps replace the exponent $1/2$ by $1$. If instead of considering
$S_{+}$ and $S_{-}$, we work directly with
$$S(R)=\sum_{n\leq R}\ind_{[F(\sigma_n)>0]}\ind_{[F(\sigma_{n+1})<0]},$$
then the number of zeros can be lower bounded directly by a rescaled quantity involving $S(R)$. The problem that we could not solve is to compute the correlations of $\{\ind_{[F(\sigma_n)>0]}\ind_{[F(\sigma_{n+1})<0]}\}_{n\geq 1}$, since this involves quadruple integrals and not so nice as in Lemma \ref{lemma correlacao de eventos}. We hope to investigate this in another occasion.

\subsection{General random Dirichlet series} It is interesting to observe from formula \eqref{formula zeros alpha} that a number $\lambda>0$ such that
\begin{equation*}
	\pi(x)=(\lambda+o(1))x(\log x)^\alpha
\end{equation*}
has no effect in the asymptotics of $\EE N_\alpha(T,U)$. 

Another interesting remark comes from the fact that we could deal with a slight more general random Dirichlet series if we allow to put extra weights $\{a_p\geq 0:p\in\mathcal{P}\}$:
$$F(\sigma)=\sum_{p\in\mathcal{P}} \frac{a_p X_p}{p^\sigma}.$$
All we need to do is to make regularity assumptions on the partial sums
$$\pi_*(x):=\sum_{p\leq x}a_p^2.$$

In this case, formula \eqref{formula zeros alpha} remains valid if we replace $\zeta_\alpha$ by 
$$\zeta_*(s):=\sum_{p\in\mathcal{P}}\frac{a_p^2}{p^s}.$$

The results of Theorem \ref{teorema alpha} remains unchanged if we worked out with assumptions on $\pi_*(x)$ instead of $\pi(x)$, since all what matters is the behaviour of $\zeta_*(s)$ around its singularity.

An interesting example comes when we consider $\mathcal{P}=\mathbb{N}$ and $a_n=\sqrt{\tau(n)}$, where $\tau(n)$ is the number of positive divisors of $n$. In this case, $\zeta_*(s)=\zeta^2(s)$ and the expected number of zeros will be just $\sqrt{2}$ times $\EE N(T,\infty)$ given by Theorem \ref{teorema principal}.

\noindent \textbf{Acknowledgements}. We are warmly thankful to Roberto Imbuzeiro and to user Iosif Pinelis from \textit{mathoverflow} for helping us with Lemma \ref{lemma correlacao de eventos}. This project is supported by CNPq, grant Universal number 403037/2021-2 and was completed while the first author was a visiting professor at Aix-Marseille Université. He is thankful to CNPq for supporting this visit with the grant Bolsa PDE, number 400010/2022-4 (200121/2022-7). The revision of this paper was done after the authors were granted by FAPEMIG. We thank FAPEMIG for supporting us with `Universal', grant number APQ-00256-23. 

{\small{\sc \noindent
Departamento de Matem\'atica, Universidade Federal de Minas Gerais (UFMG), Av. Ant\^onio Carlos, 6627, CEP 31270-901, Belo Horizonte, MG, Brazil.} \\
\textit{Email address:} \verb|aymone.marco@gmail.com| }

{\small{\sc \noindent
Departamento de Matem\'atica, Universidade Federal do Rio Grande do Sul (UFRGS), Av. Bento Gonçalves, 9500 Prédio 43-111 – Agronomia
91509-900 Porto Alegre, RS, Brazil.} \\
\textit{Email address:} \verb|sfrometa@gmail.com|, \verb|ricardomisturini@gmail.com|}

\end{document}